\documentclass[11pt,a4paper]{amsart}
\usepackage{amsmath, amsthm, amsfonts, amssymb, amscd}
\usepackage[bookmarks=false]{hyperref}

\newtheorem{theo}{Theorem}[section]
\newtheorem{prop}[theo]{Proposition}
\newtheorem{lem}[theo]{Lemma}

\newtheorem{letterthm}{Theorem}

\newtheorem{lettercor}[letterthm]{Corollary}

\theoremstyle{definition} 

\newtheorem*{defs*}{Definition}

\newtheorem*{claim}{Claim}

\newcommand{\R}{\mathbb{R}}
\newcommand{\C}{\mathbb{C}}
\newcommand{\Z}{\mathbb{Z}}

\newcommand{\N}{\mathbb{N}}

\newcommand{\Aut}{\operatorname{Aut}}

\newcommand{\ovt}{\mathbin{\overline{\otimes}}}

  \title{Spectral gap characterization of full type $\mathrm{III}$ factors}
  \author{Amine Marrakchi}
  \address{\'Ecole Normale Sup\'erieure \\ 45 rue d'Ulm 75230 Paris Cedex 05 \\ France}
\address{Laboratoire de Math\'ematiques d'Orsay \\ Universit\'e Paris-Sud \\ Universit\'e Paris-Saclay \\ 91405 Orsay \\ France}
\email{amine.marrakchi@ens.fr}
\thanks{A. Marrakchi is supported by ERC Starting Grant GAN 637601}

\begin{document}
	
\begin{abstract}
We give a spectral gap characterization of fullness for type $\mathrm{III}$ factors which is the analog of a theorem of Connes in the tracial case. Using this criterion, we generalize a theorem of Jones by proving that if $M$ is a full factor and $\sigma : G  \rightarrow \mathrm{Aut}(M)$ is an outer action of a discrete group $G$ whose image in $\mathrm{Out}(M)$ is discrete then the crossed product von Neumann algebra $M \rtimes_\sigma G$ is also a full factor. We apply this result to prove the following conjecture of Tomatsu-Ueda: the continuous core of a type $\mathrm{III}_1$ factor $M$ is full if and only if $M$ is full and its $\tau$ invariant is the usual topology on $\R$.
\end{abstract}	
	

  \maketitle

	\bibliographystyle{plain}

\section{Introduction}
We say that a $\mathrm{II}_1$ factor $M$ is \emph{full} (or equivalently that $M$ does not have property ($\Gamma$)) if for every uniformly bounded net $(x_i)_i$ in $M$ such that $|| x_i a - a x_i ||_2 \rightarrow 0$ for all $a \in M$ we can find a net $\lambda_i \in \C$ such that $|| x_i - \lambda_i ||_2 \rightarrow 0$. Fullness can be intuitively understood as a \emph{strong factoriality} property: it says that an element which almost commutes with all elements of $M$ is almost trivial. This extensively studied property played an important role since the beginning of the theory of von Neumann algebras: Murray and von Neumann first used it to show the existence of a $\mathrm{II}_1$ factor which is not hyperfinite. Indeed, while a hyperfinite factor is never full, Murray and von Neumann were able to prove, using their so-called $14 \varepsilon$-lemma, that the $\mathrm{II}_1$ factor $\mathcal{L}(\mathbb{F}_2)$ generated by the free group on $2$ generators is full. In fact, their proof shows that the following, a priori much stronger, \emph{spectral gap} property is satisfied:
\[ \exists C >0, \; \forall x \in \mathcal{L}(\mathbb{F}_2), \; || x- \tau(x) ||^{2}_2 \leq C(|| xa-ax||^{2}_2+|| xb-bx||^{2}_2) \]
where $a$ and $b$ are the canonical generators of $\mathbb{F}_2$. It is a spectral gap property because it says that $0$ is a simple isolated point in the spectrum of the positive operator $ |\lambda(a)-\rho(a)|^2+|\lambda(b)-\rho(b)|^2 \in B(\ell^2(\mathbb{F}_2))$ where $\lambda$ and $\rho$ are the left and right regular representations of $\mathbb{F}_2$. Later on, Connes proved that, in fact, \emph{every} full $\mathrm{II}_1$ factor satisfies such a spectral gap property \cite[Theorem 2.1]{connes1976classification}. More precisely, he proved that for every full separable $\mathrm{II}_1$ factor $M$, there exist a family $a_1,\dots, a_k \in M$ and a constant $C > 0$ such that for all $x \in M$ we have
\[ || x- \tau(x) ||^{2}_2 \leq C \sum_k || xa_k-a_k x||^{2}_2 \]
This deep theorem of Connes was one of the key steps in his proof of the uniqueness of the injective $\mathrm{II}_1$ factor. One of its most striking corollaries is the fullness of the tensor product $M \ovt N$ of two full $\mathrm{II}_1$ factors $M$ and $N$ \cite[Corollary 2.3]{connes1976classification}. This theorem is also very remarkable because such an equivalence between a spectral gap property and its weaker "bounded" counterpart fails in many other similar situations (see \cite{popa2012classification}[Remark 2.2]). 

In this paper, we are interested in full type $\mathrm{III}$ factors. The correct definition of fullness for arbitrary factors, as given by Connes in \cite{ConnesAlmostPeriodic}, is the following one: a factor $M$ is said to be \emph{full} if for every uniformly bounded net $(x_i)_i$ in $M$ such that $x_i \varphi - \varphi x_i \rightarrow 0$ for all $\varphi \in M_*$ we can find a net $\lambda_i \in \C$ such that $x_i - \lambda_i \rightarrow 0$ in the strong topology. Then our main result is the following type $\mathrm{III}$ analog of the theorem of Connes.

\begin{letterthm} \label{main}
Let $M$ be a full $\sigma$-finite type $\mathrm{III}$ factor. Then there exist a faithful normal state $\varphi$ on $M$, a family $\xi_1, \cdots, \xi_n \in L^{2}(M)^{+}$ with $\xi_k^{2} \leq \varphi$ for all $k$ such that for all $x \in M$ we have
\[ ||x-\varphi(x)||_\varphi^2 \leq  \sum_k || x \xi_k-\xi_k  x ||^2  \]
\end{letterthm}

We also give an application of Theorem \ref{main} to fullness of crossed products which generalizes a theorem of Jones \cite{jones1982central}.

\begin{letterthm} \label{crossed_product}
Let $M$ be a full factor and $G$ a discrete group. Let $\sigma : G \rightarrow \mathrm{Aut}(M)$ be an outer action whose image in $\mathrm{Out}(M)$ is discrete. Then the crossed product $M \rtimes_\sigma G$ is also a full factor.
\end{letterthm}

Using this theorem, we give a positive answer to a conjecture of Tomatsu-Ueda \cite{tomatsu2014characterization} which characterizes the fullness of continuous cores of type $\mathrm{III}_1$ factors in terms of their $\tau$-invariant introduced in \cite{ConnesAlmostPeriodic}. The only if part was first observed by Shlyakhtenko in \cite{shlyakhtenko2004classification}. The if part is much harder and was obtained by Tomatsu-Ueda in particular cases, namely for free product factors as well as generalized Bernoulli crossed products \cite{tomatsu2014characterization}. Using Theorem \ref{crossed_product} in combination with the key idea of \cite[Lemma 6]{tomatsu2014characterization}, we can prove it in full generality.

\begin{lettercor} \label{ueda}
Let $M$ be a factor of type $\mathrm{III}_1$ with separable predual. Then its continuous core $c(M)$ is full if and only if $M$ is full and the canonical morphism $\delta : \R \rightarrow \mathrm{Out}(M)$ is a homeomorphism on its range (i.e \ the $\tau$-invariant of $M$ is the usual topology).
\end{lettercor}

The article is organized as follows. Section 2 is devoted to notations and preliminaries. In particular, we recall the constructions of the \emph{Groh-Raynaud ultraproduct} of von Neumann algebras. In Section 3, we recall the commutation techniques developped in \cite{connes_stormer}. In Section 4, we first give a new proof of the spectral gap theorem in the $\mathrm{II}_1$ case. Our proof is quite different and perhaps easier than Connes's original proof, even though it is inspired from it. The main difference is that we use Groh-Raynaud ultraproducts instead of \emph{singular states}. After that, we adapt this proof to the type $\mathrm{III}$ case in order to prove Theorem \ref{main}. Finally, in the last section of the paper, we apply our result to crossed products following the lines of \cite{jones1982central}.

\subsubsection*{Acknowledgment}
We are very grateful to our advisor Cyril Houdayer for attracting our attention to this problem and for his help and suggestions throughout this work. We also thank Yoshimichi Ueda for explaining to us \cite[Lemma 6]{tomatsu2014characterization} and for his useful comments.

\tableofcontents

\section{Preliminaries}

Let $M$ be a von Neumann algebra. We denote by $M_*$ its predual. We say that $M$ is \emph{$\sigma$-finite} if it admits a faithful normal state. We denote by $M^+$ and $M_*^+$ the positive parts of $M$ and $M_*$. We denote by $\widehat{M_*^+}$ the set of all semifinite normal weights on $M$ and by $\widehat{M^+}$ the set of all positive operators affiliated with $M$. The strong topology on $M$ is the topology induced by the family of semi-norms $x \mapsto ||x ||_\varphi :=\varphi(x^*x)^{\frac{1}{2}}$ for $\varphi \in M_*^+$ and the $*$-strong topology is the topology induced by the family of semi-norms $x \mapsto ||x ||_\varphi +||x^* ||_\varphi $ for $\varphi \in M_*^+$. We denote by $\mathcal{U}(M)$ the group of unitaries of $M$ and by $\mathcal{P}(M)$ the lattice of all projections of $M$. Both are implicitely equipped with the restriction of the strong topology. We denote by $\mathrm{Aut}(M)$ the group of automorphisms of $M$. Every $\theta \in \mathrm{Aut}(M)$ induces by the precomposition $\varphi \mapsto \theta(\varphi):=\varphi \circ \theta^{-1}$ an automorphism of $M_*$. Then $\mathrm{Aut}(M)$ is equipped with the topology of pointwise norm convergence on $M_*$. Hence a net $\theta_i \in \mathrm{Aut}(M), i \in I$ converges to $\theta$ if and only if $\theta_i(\varphi) \rightarrow \theta(\varphi)$ in norm for every $\varphi \in M_*$. We denote by $\mathrm{Ad}(u) : \mathcal{U}(M) \rightarrow \mathrm{Aut}(M)$ the continuous group homomorphism given by $\mathrm{Ad}(u)(x)=uxu^*$. The image of this homomorphism, the group of inner automorphisms, is denoted $\mathrm{Inn}(M)$ and the quotient group $\mathrm{Aut}(M)/\mathrm{Inn}(M)$ is denoted $\mathrm{Out}(M)$. The quotient map is denoted $\epsilon : \mathrm{Aut}(M) \rightarrow \mathrm{Out}(M)$.

\subsection{Modular theory}
Let $M$ be a von Neumann algebra. We denote by $c(M)$ its \emph{canonical core} \cite[Chapter $\mathrm{XII}.6$]{TakesakiII}. It is a von Neumann algebra canonically associated to $M$ with the following properties. First, $c(M)$ contains $M$ as a von Neumann subalgebra. There is a canonical faithful semifinite normal trace $\tau$ on $c(M)$ and a canonical continuous action of $\R^*_+$ on $c(M)$, called the \emph{non-commutative flow of weights} and denoted by $\lambda \mapsto \theta_\lambda \in \mathrm{Aut}(c(M))$, such that $\theta_\lambda(\tau)=\frac{1}{\lambda} \tau$ for all $\lambda > 0$. The fixed point subalgebra $\{ x \in c(M) \mid \forall \lambda \in \R^*_+, \; \theta_\lambda(x)=x \}$ is exactly $M$. By \cite{haagerup1979lp}, there is a canonical bijection $ \iota : \widehat{M_*^+} \rightarrow \{ h \in \widehat{c(M)^+} \mid \forall \lambda \in \R^*_+, \; \theta_\lambda(h)=\lambda h \}$, called the \emph{Haagerup correspondence}, which is additive, positively homogeneous and satisfies $\iota(a\varphi a^*)=a \iota(\varphi)a^*$ for all $a \in M$. Hence, by a harmless abuse of notation we will identify $\varphi$ and $\iota(\varphi)$ and view every semifinite normal weight on $M$ as a positive operator affiliated with $c(M)$. If $\varphi$ is a faithful semifinite normal weight on $M$ then $(\varphi^{{\rm i} t})_{ t \in \R}$ is a one-parameter group of unitaries in $c(M)$ which satisfies $\theta_\lambda(\varphi^{{\rm i} t})=\lambda^{{\rm i} t}\varphi^{{\rm i} t}$ for all $t \in \R$ and $\lambda \in \R^*_+$. Therefore if $x \in M$, then $\theta_\lambda(\varphi^{{\rm i} t}x\varphi^{-{\rm i} t})=\varphi^{{\rm i} t}x\varphi^{-{\rm i} t}$ for all $\lambda$ which means that $\varphi^{{\rm i} t}x\varphi^{-{\rm i} t} \in M$. Hence we have a one-parameter group of automorphisms $t \mapsto \sigma_t^\varphi \in  \mathrm{Aut}(M)$, called the \emph{modular flow} of $\varphi$, which is given by $\sigma_t^{\varphi}(x)= \varphi^{{\rm i} t}x\varphi^{-{\rm i} t}$. An element $x \in M$ is said to be \emph{$\varphi$-analytic} if the map $t \in \R \mapsto \sigma_{t}^\varphi(x) \in M$ can be extended to a holomorphic map defined on the whole complex plan. In that case, this map is denoted by $z \in \C \mapsto \sigma_{z}^\varphi(x) \in M$. The $\varphi$-analytic elements form a dense $*$-algebra in $M$ \cite[Lemma $\mathrm{VIII}.2.3$]{TakesakiII}. From a given faithful semifinite normal weight $\varphi$ and its modular flow $\sigma^{\varphi} : \R \rightarrow \mathrm{Aut}(M)$, one can reconstruct $c(M)$ as the crossed product $M \rtimes_{\sigma^{\varphi}} \R$. If $\psi$ is another faithful semifinite normal weight on $M$, then the unitary $u_t=\psi^{{\rm i} t} \varphi^{-{\rm i} t}$, called the \emph{Connes-Radon-Nikodym cocycle}, satisfies $\sigma_t^\psi = \mathrm{Ad}(u_t) \circ \sigma_t^\varphi$ and $u_t$ is fixed by $(\theta_\lambda)_{\lambda \in \R^*_+}$ hence $u_t \in M$. In particular, the equivalence class $\delta(t)=[ \sigma_t^\varphi] \in \mathrm{Out}(M)$ is independent of the choice of $\varphi$ which means that we have a canonical homomorphism $\delta : \R \rightarrow \mathrm{Out}(M)$. 

\subsection{The standard form} 
 Let $M$ be a von Neumann algebra. Let $c(M)$ be the core of $M$ with its canoncal trace $\tau$. Then the set of all \emph{$\tau$-measurable} operators affiliated with $c(M)$ form a nice topological $*$-algebra obtained as the completion of $c(M)$ for the \emph{measure topology} as explained in \cite[Chapter $\mathrm{IX}.2$]{TakesakiII}. Hence we can multiply them as explained in \cite[Theorem $\mathrm{IX}.2.2$]{TakesakiII} and for this, we do not use their representation as closed densely defined operators (\cite[Definition $\mathrm{IX}.2.4$]{TakesakiII}) so there are no domain issues. By\cite{haagerup1979lp}, a semifinite normal weight $\varphi$ on $M$ is finite (i.e.\ $\varphi(1) < +\infty$) if and only if it is \emph{$\tau$-measurable} as a positive operator affiliated with $c(M)$. And more generally, the predual $M_*$ can be identified with the space $L^1(M)$ of all $\tau$-measurable operators $h$ affiliated with $c(M)$ such that $\theta_\lambda(h)=\lambda h$ for all $\lambda > 0$. For every $\varphi \in M_*=L^1(M)$ we let $\langle \varphi \rangle :=\varphi(1)$. Following \cite{haagerup1979lp}, we let $L^2(M)$ denote the space of all $\tau$-measurable operators $\xi$ affiliated with $c(M)$ such that $\theta_\lambda(\xi)=\lambda^{1/2}\xi$ for all $\lambda > 0$. Observe that if $\varphi \in M_*^+$ then $\varphi^{1/2} \in L^2(M)$. If $\xi \in L^2(M)$ and $\xi = u |\xi |$ is its polar decomposition as a $\tau$-measurable operator, then one has $u \in M$ and $|\xi | \in L^2(M)$. If $\xi, \eta \in L^2(M)$ then one has $\xi \eta \in L^1(M)=M_*$ and $\langle \xi \eta \rangle = \langle \eta \xi \rangle$. Hence, one can define an inner product on $L^2(M)$ by 
\[ \langle \xi , \eta \rangle := \langle \xi \eta^* \rangle= \langle \eta^* \xi \rangle \]
and this turns $L^2(M)$ into a Hilbert space. We have $|| \xi \eta ||_1 \leq || \xi ||_2 || \eta ||_2$ for all $\xi, \eta \in L^2(M)$. If $x \in M$ then $x \xi \in L^2(M)$ for all $\xi \in L^2(M)$ and $|| x \xi ||_2 \leq ||x ||_\infty ||\xi ||_2$. Therefore we can define a bounded operator in $B(L^2(M))$ by $\lambda(x) : \xi \mapsto x \xi$. Similarly, we define a bounded operator $\rho(x) : \xi \mapsto \xi x$. Then $\lambda : M \rightarrow B(L^2(M))$ is a faithful normal representation and $\rho : M \rightarrow B(L^2(M))$ is a faithful normal antirepresentation and we have $\lambda(M)'=\rho(M)$. We define the conjugate linear isometry $J$ on $L^2(M)$ by $J : \xi \mapsto \xi^*$ and the set of positive elements of $L^2(M)$ is denoted $L^2(M)^+$. The quadruple $(\lambda(M),L^2(M),J,L^2(M)^+)$ is the \emph{standard form} of $M$ \cite[Theorem 1.21]{haagerup1979lp}. 

Finally, to familiarize the reader with this quite unusual point of view where we see $L^{2}(M)$ and $L^{1}(M)=M_*$ as subsets of the $*$-algebra of all $\tau$-measurable operators affiliated with $c(M)$, we prove the following lemma which already appears in \cite{kosaki1984applications}[Lemma 2.3 and 2.4].

\begin{lem} \label{strongly_commute}
Let $M$ be a von Neumann algebra with two faithful normal states $\varphi$ and $\psi$. If $\varphi \psi = \psi \varphi$ (as $\tau$-measurable operators) then $\varphi$ and $\psi$ strongly commute as unbounded self-adjoint operators (i.e\ all their spectral projections commute) and in particular we have $\psi^{it}\varphi \psi^{-it}=\varphi$ for all $t \in \R$.
\end{lem}
\begin{proof}
Since $\varphi$ and $\psi$ are positive, $(1+\varphi)$ and $(1+\psi)$ have two bounded inverses $(1+\varphi)^{-1}$ and $(1+\psi)^{-1}$ in $c(M)$. Since $\varphi\psi = \psi \varphi$ we have $(1+\varphi)(1+\psi)=(1+\psi)(1+\varphi)$ and by multipliying by $(1+\psi)^{-1}$ and $(1+\varphi)^{-1}$ (no domains involved here, we just use the $*$-algebra structure as defined in \cite[Theorem $\mathrm{IX}.2.2$]{TakesakiII}), we get $(1+\varphi)^{-1}(1+\psi)^{-1}=(1+\psi)^{-1}(1+\varphi)^{-1}$. Since this two self-adjoint bounded operators commute, their spectral projections also commute. But the spectral projections of $(1+\varphi)^{-1}$ (resp. $(1+\psi)^{-1}$) are the same as the spectral projections of $\varphi$ (resp. $\psi$) and this proves the lemma.
\end{proof}

\subsection{Ultraproducts}
In this section, we recall the definitions of ultraproducts of von Neumann algebras and we fix the notations. We refer to \cite{AndoHaagUltra} for the proofs. For technical reasons, we need to use ultrafilters on arbitrary directed sets and not only on $\N$ (even if we assume that $M$ has separable predual in Theorem \ref{main}). The proofs of the few properties that we will need from \cite{AndoHaagUltra} work verbatim in this more general setting. 

Let $I$ be a directed set and $\omega$ a cofinal ultrafilter on $I$ (cofinal means that $\{ j \in I \mid j \geq i \} \in \omega$ for every $i \in I$). Let $(E,||\cdot ||)$ be a Banach space. We define a new Banach space $(E, || \cdot ||)^\omega$ called the \emph{ultraproduct} of $(E, || \cdot ||)$ with respect to $\omega$. It is the quotient of $\ell^\infty(I,E)$ by the closed subspace 
\[ \{ (x_i)_{i \in I} \in \ell^\infty(I,E) \mid \lim_{i \rightarrow \omega } ||x_i ||=0 \} \]
If $(x_i)_{ i \in I} \in \ell^{\infty}(I,E)$ we denote by $(x_i)^\omega$ its class in $(E,||\cdot ||)^\omega$. If $x_i:=x$ is a constant net, we denote its class by $x^\omega$. The norm on $(E, || \cdot ||)^\omega$ is defined by
\[ ||(x_i)^\omega || = \lim_{i \rightarrow \omega} ||x_i || \]
When $H$ is a Hilbert space, the Banach space ultraproduct $H^\omega$ is also a Hilbert space with the following scalar product
\[ \langle (\xi_i)^\omega, (\eta_i)^\omega \rangle = \lim_{i \rightarrow \omega } \langle \xi_i , \eta_i \rangle \]
One can view $(B(H),|| \cdot ||)^\omega$ naturally as a $C^*$-subalgebra of $B(H^\omega)$ via the $*$-homomorphism
\[ (T_i)^\omega \mapsto \left( (\xi_i)^\omega \mapsto (T_i\xi_i)^\omega \right) \]
but we warn the reader that the $*$-homomorphism $T \in B(H) \mapsto T^\omega \in B(H^\omega)$ is not normal in general.

Now, fix $M$ a $\sigma$-finite von Neumann algebra. Let $A=(M,|| \cdot ||)^\omega$ be the Banach space ulraproduct of $M$ with respect to $\omega$. Then $A$ is naturally a $C^*$-algebra but it is not a von Neumann algebra in general. Let $A^{**}$ be the bidual of $A$ (which is a von Neumann algebra). Let $(M_*)^\omega$ be the Banach space ultraproduct of $M_*$. Then $(M_*)^\omega$ can be identified naturally with a closed subspace of $A^*$ via the embedding
\[ (\varphi_i)^\omega \mapsto \left( (x_i)^\omega \mapsto \lim_{i \rightarrow \omega} \varphi_i(x_i) \right) \]
Then the orthogonal of $(M_*)^\omega$ in $A^{**}$ defined by
\[ \mathfrak{J}=\{ x \in A^{**} \mid \forall \varphi \in (M_*)^\omega, \; \varphi(x)=0 \} \]
is a weak* closed ideal in the von Neumann algebra $A^{**}$ which means that the quotient $M^\omega_{GR} := A^{**}/\mathfrak{J}$ is a von Neumann algebra. It is called the \emph{Groh-Raynaud ultraproduct} of $M$ (with respect to $\omega$). By construction, the predual of $M^\omega_{GR}$ is exactly $(M_*)^\omega$ and $M^\omega_{GR}$ contains the Banach space ultraproduct $A=(M,|| \cdot ||)^\omega$ as a dense $C^*$-subalgebra. The $*$-homomorphism $x \in M \mapsto x^\omega \in M^\omega_{GR}$ is not normal in general so $M$ is not a von Neumann subalgebra of $M^\omega_{GR}$. The von Neumann algebra $M^\omega_{GR}$ is very large (not separable and not even $\sigma$-finite in general). The main interest in this ultraproduct comes from the fact that, as explained in \cite{AndoHaagUltra} and \cite{raynaud2002ultrapowers}, there is a natural identification $L^2(M_{GR}^\omega)=L^2(M)^\omega$ in such a way that $\lambda((x_i)^\omega)$ is identified with $(\lambda(x_i))^\omega$ and $\rho((x_i)^\omega)$ is identified with $(\rho(x_i))^\omega$ for all $(x_i)^\omega \in (M, ||\cdot ||)^\omega$. The positive cone $L^2(M_{GR}^\omega)^+$ is identified with $\{ (\xi_i)^\omega \mid \xi_i \in L^2(M)^+ \}$ and the conjugation is given by $((\xi_i)^\omega)^*=(\xi_i^*)^\omega$.

Now choose a faithful normal state $\varphi \in M_*^+$. Then we have $\varphi^\omega \in (M^\omega_{GR})_*^+$ but $\varphi^{\omega}$ is not faithful in general. Let $e$ be the support of $\varphi^{\omega}$ in $M^\omega_{GR}$. The projection $e$ does not depend on the choice of $\varphi$ and the corner $e(M^\omega_{GR})e$ is called the \emph{Ocneanu ultraproduct} of $M$ (relatively to $\omega$). It is denoted by $M^{\omega}$. For all $x \in M$ we have $x^{\omega}e=ex^{\omega}$ and the $*$-homomorphism $x \in M \mapsto x^{\omega}e \in M^{\omega}$ is normal. Hence we may sometimes abuse the notation and view $M$ as a von Neumann subalgebra of $M^{\omega}$. In this case, one can define a canonical faithful normal conditional expectation $E^{\omega} : M^{\omega} \rightarrow M$ by the formula 
\[ E^{\omega}((x_i)^{\omega})=\lim_{i \rightarrow \omega} x_i  \; \text{ in the weak* topology} \]
Finally, for the Ocneanu ultraproduct we have $L^2(M^\omega)=e(L^2(M)^\omega) e$.

\subsection{Full factors}
Let $M$ be a factor. A \emph{centralizing net} in $M$ is a bounded net $x_i \in M, i \in I$ such that $||x_i \varphi - \varphi x_i|| \rightarrow 0$ for all $\varphi \in M_*$. We say that $(x_i)_{i \in I}$ is \emph{trivial} if there exists a net $\lambda_i \in \C, i \in I$ such that $x_i - \lambda_i \rightarrow 0$ in the strong topology. We say that the factor $M$ is \emph{full} if every centralizing net is trivial. It is not hard to check (see \cite{ConnesAlmostPeriodic}) that $M$ is full if and only if the map $\mathrm{Ad}: \mathcal{U}(M) \rightarrow \mathrm{Aut}(M)$ is open on its range. This means that if some net of unitaries $u_i \in \mathcal{U}(M), i \in I$ satisfies $\mathrm{Ad}(u_i) \rightarrow 1$ then there exists a net of scalars of modulus one $z_i \in \mathbb{U}=\ker \mathrm{Ad}$ such that $z_iu_i \rightarrow 1$. In particular, as it is shown in \cite{ConnesAlmostPeriodic}, if $M$ is full then $\mathrm{Inn}(M)$ is closed in $\mathrm{Aut}(M)$ (the converse is also true when $M$ has separable predual but is probably not true in general). 

Now, suppose that $M$ is $\sigma$-finite.  For a cofinal ultrafilter $\omega$ on some directed set $I$, consider the von Neumann algebra $M' \cap M^\omega$. Since $M$ is a factor, the faithful normal conditional expectation $E^{\omega} : M^{\omega} \rightarrow M$ restricts to a faithful normal state $\Psi  : M' \cap M^{\omega} \rightarrow \C$ called the \emph{Golodets state} (see \cite{AndoHaagUltra}). Since $\varphi^{\omega}=\varphi \circ E^{\omega}$, we have $\Psi=\varphi^{\omega}_{\mid M' \cap M^{\omega}}$ for any normal state $\varphi \in M_*^{+}$. The centralizer of $\Psi$ in $M' \cap M^{\omega}$ is denoted by $M_\omega$ and is called the \emph{asymptotic centralizer} of $M$ (relatively to $\omega$). If $x \in (M, || \cdot ||)^{\omega}$ satisfies $x \varphi^{\omega}=\varphi^{\omega}x$ for all $\varphi \in M_*$ then $xe=ex \in M_\omega$ and all elements of $M_\omega$ are of this form. It is easy to check (see \cite{ConnesAlmostPeriodic}) that $M$ is full if and only if $M_\omega = \C$ for every cofinal ultrafilter $\omega$ on every directed set $I$. Note that when $M$ has separable predual and $\omega$ is an ultrafilter on $I=\N$ then we have $M_\omega= \C$ if and only if $M' \cap M^{\omega}=\C$ by \cite{AndoHaagUltra}[Theorem 5.2].

\section{Skew information}
In this section we recall some of the techniques that were developped in \cite{connes1976classification}, \cite{connes_stormer} and \cite{connes1985factors}. They involve a subtle analysis of a functional, called the \emph{skew information}, which is defined by
\[ I(x,\xi)=\frac{1}{2}|| x \xi - \xi x||^2 \]
for $x$ in a von Neumann algebra $M$ and $\xi \in L^2(M)^+$.

\begin{prop} \label{skew}
We have the following properties:
\begin{enumerate}
\item $|| x \varphi-\varphi x ||^{2} \leq 8 \varphi(1)I(x,\varphi^{1/2})$ for all $x \in M$ and $\varphi \in M_{*}^{+}$.
\item $I(|x|,\xi) \leq I(x,\xi)$ for all $x \in M$ with $x^{*}x=xx^{*}$ and all $\xi \in L^2(M)^+$.
\item $I(|x|,\xi) + I(|x^*|,\eta)  \leq || x \xi-\eta x||^2$ for all $x \in M$, $\xi, \eta \in L^2(M)^+$.
\item $I(p,\xi)=||p\xi p^{\perp}||^2$ for all $p \in \mathcal{P}(M)$ and $\xi \in L^2(M)^+$.
\item $I(p+q,\xi) = I(p,q^{\perp} \xi q^{\perp})+I(q,p^{\perp}\xi p^{\perp})$ for all $p, q \in \mathcal{P}(M)$ with $pq=0$ and all $\xi \in L^2(M)^+$.
\end{enumerate}
\end{prop}
\begin{proof}
$(1)$ We have $|| x \varphi - \varphi x || \leq || x\varphi - \varphi^{1/2}x\varphi^{1/2}|| + || \varphi^{1/2}x\varphi^{1/2} - \varphi x ||$. Thus
\[ || x \varphi - \varphi x ||  \leq   || x \varphi^{1/2}-\varphi^{1/2} x || \cdot || \varphi^{1/2} || + || \varphi^{1/2} || \cdot || x \varphi^{1/2}-\varphi^{1/2} x || \]
Since $|| \varphi^{1/2} || =\varphi(1)^{1/2}$. We get $|| x \varphi - \varphi x ||  \leq 2 \varphi(1)^{1/2} || x \varphi^{1/2}-\varphi^{1/2} x ||$ as we wanted.

$(2)$. Since $x$ is normal, the operators $\lambda(x)$ and $\rho(x)$ generate a commutative von Neumann algebra in $B(L^{2}(M))$. Hence by the classical triangle inequality we have $| \, |\lambda(x)|-|\rho(x)| \,|^{2} \leq |\lambda(x)-\rho(x)|^{2}$. Hence by applying the positive linear form $\langle \cdot \xi, \xi \rangle$ we get $|| (|\lambda(x)|-|\rho(x)|)\xi ||^{2} \leq || (\lambda(x)-\rho(x))\xi ||^{2}$ which means that $|||x|\xi- \xi |x| ||^{2} \leq || x \xi - \xi x ||^{2}$.

$(3)$. For $N=M_2(\C) \otimes M=M_2(M)$ consider
\[ y=
\begin{pmatrix} 
0 & x^* \\
x & 0
\end{pmatrix} \in N \text{ and } \alpha= 
\begin{pmatrix}
\xi & 0 \\
0 & \eta 
\end{pmatrix} \in L^2(N)^+
\]
Then we have 
\[ y\alpha-\alpha y = \begin{pmatrix}
0 & x^*\eta-\xi x^* \\
x\xi-\eta x & 0
\end{pmatrix}
\]
so that $I(y,\alpha)=||x \xi- \eta x||^2$. We also have
\[ |y|=\begin{pmatrix}
|x| & 0 \\
0 & |x^*|
\end{pmatrix}\]
which means that
\[ |y|\alpha-\alpha |y| = \begin{pmatrix}
|x| \xi - \xi |x| & 0 \\
0 & |x^*|\eta - \eta |x^*|
\end{pmatrix}
 \]
so that $I(|y|,\alpha) = I(|x|,\xi)+I(|x^*|,\eta)$. As $y$ is self-adjoint, the conclusion follows from $(2)$.

$(4)$. We have $||p\xi-\xi p||^2=||p\xi p^{\perp}-p^{\perp}\xi p||^2=||p\xi p^{\perp}||^2+||p^{\perp}\xi p||^2=2||p\xi p^{\perp}||^2$ because $p\xi p^{\perp}$ and $p^{\perp}\xi p$ are orthogonal.

$(5)$. By $(4)$, we have $I(p+q,\xi)=||(p+q)\xi p^{\perp}q^{\perp}||^2=||p\xi p^{\perp}q^{\perp}||^2+||q\xi p^{\perp}q^{\perp}||^2=||pq^{\perp}\xi q^{\perp}p^{\perp}||^2+||qp^{\perp}\xi p^{\perp}q^{\perp}||^2=I(p,q^{\perp} \xi q^{\perp})+I(q,p^{\perp}\xi p^{\perp})$.
\end{proof}

The following crucial lemma was obtained in \cite[Corollary 3]{connes_stormer} but the main idea behind it already appears in \cite{connes1976classification}. See also \cite[Corollary $\mathrm{IX}.1.23$]{TakesakiII}.

\begin{lem} \label{spike}
Let $M$ be a von Neumann algebra and let $\xi \in L^2(M)^+$. For every $\varepsilon > 0$ and every non-zero $x \in M^{+}$ which satisfies $I(x,\xi) \leq \varepsilon ||x\xi||^2$ there exists $c > 0$ such that $p=1_{[c,+\infty)}(x) \neq 0$ and $I(p,\xi) \leq 4\sqrt{\varepsilon} ||p\xi ||^2$.
\end{lem}

Finally, we mention a technique that will allow us to apply Proposition \ref{skew} and Lemma \ref{spike} to finite family of vectors simultaneously. Let $\xi_1, \dots, \xi_n \in L^2(M)^+$ be a finite family of vectors. Let $M^{(n)}=\bigoplus_{k=1}^n M$. Define $\xi=(\xi_1,\dots, \xi_n) \in L^2(M^{(n)})^+$ and view $M$ as a von Neumann subalgebra of $M^{(n)}$ via the diagonal inclusion $x \mapsto (x,x,\dots,x)$. Then we have 
\[ I(x,\xi)=\sum_{k=1}^n I(x,\xi_k)\]
 If we let 
\[ \varphi=\xi_1^2+\dots + \xi_n^2 \in M_*^+ \]
then we have 
 \[ \varphi(x)=\sum_{k=1}^n \langle x \xi_k, \xi_k \rangle = \langle x \xi, \xi \rangle \]
 and 
\[ ||x||_\varphi^2=\sum_{k=1}^n ||x \xi_k||^2=|| x \xi ||^2 \]
and we will say that $(\xi_1,\dots,\xi_n)$ is a \emph{partition} of $\varphi$.

In relation with this remark, we note the following useful fact: the condition $|\eta|^2 \leq \psi$ for $\eta \in L^2(M)$ and $\psi \in M_*^+$ is equivalent to the existence of $a \in M$ with $||a|| \leq 1$ such that $\eta=a \psi^{1/2}$ (indeed $\eta \psi^{-1/2}$ defines a bounded operator in $c(M)$ which is invariant by $(\theta_\lambda)_{\lambda \in \R^*_+}$). In particular, we have $M \psi^{1/2}=\{ \eta \in L^2(M) \mid \exists \lambda > 0, \; | \eta |^2 \leq \lambda \psi\}$ and $\psi^{1/2}M \cap M \psi^{1/2}$ is the linear span of $\{ \eta \in L^2(M)^+ \mid \eta^2 \leq \psi \}$.

\section{Spectral gap characterization of fullness}

In this section we first give a different proof of the spectral gap theorem in the $\mathrm{II}_1$ case using the Groh-Raynaud ultraproduct and then we adapt it to the type $\mathrm{III}$ case in order to prove Theorem \ref{main}. 

\begin{theo}[{\cite[Theorem 2.1]{connes1976classification}}] \label{gap_connes}
Let $M$ be a full $\mathrm{II}_1$ factor. Then there exist a family $a_1,\dots, a_n \in M$ such that for all $x \in M$ we have
\[ || x- \tau(x) ||^{2}_2 \leq \sum_k || xa_k-a_k x||^{2}_2 \]
\end{theo}
\begin{proof}
First we fix the following notation: for all $x \in M$, we let $\widehat{x}:=x\tau^{1/2}=\tau^{1/2}x \in L^{2}(M)$. Now, suppose, by contradiction, that we can find a net $(x_i)_{i \in I}$ in $M$ (not necessarily bounded) such that: 
\begin{itemize}
\item $|| x_i  ||_2=1$ for all $i$.
\item $\tau(x_i)=0$ for all $i$.
\item $|| x_i a - a x_i ||_2 \rightarrow 0$ for every $a \in M$.
\end{itemize}
Take $\omega$ a cofinal ultrafilter on $I$ and let $\alpha=(\widehat{x_i})^{\omega} \in L^{2}(M)^{\omega}=L^{2}(M_{GR}^{\omega})$. Then we have
\begin{itemize}
\item $||\alpha ||=1$
\item $\alpha \perp \widehat{1}^{\omega}$.
\item $a^{\omega}\alpha=\alpha a^{\omega}$ for all $a \in M$. 
\end{itemize}
Note also that since $x \mapsto \langle x^{\omega} \alpha, \alpha \rangle$ is a tracial state on $M$ then it must coincide with $\tau$.

Now, let $e=\mathrm{supp}(\tau^{\omega}) \in  \mathcal{Z}(M_{GR}^{\omega})$ so that $eM_{GR}^{\omega}=M^{\omega}$. We have $e\alpha=\alpha e \in L^{2}(M^{\omega})$. Observe that $e\alpha$ is still $M$-central. Since $M$ is full, we have that $M' \cap M^{\omega}=\C$. Hence this implies that $e\alpha$ is proportional to $\widehat{1}^{\omega}$. But we have $\alpha \perp \widehat{1}^{\omega}$ and $e\widehat{1}^{\omega}=\widehat{1}^{\omega}$. Hence we must have $e\alpha=\alpha e=0$.

Now using the fact that $e\alpha=0$, we modify the $x_i$'s to make their supports very small. First, since $e|\alpha|=0$ and $a^{\omega}|\alpha|=|\alpha |a^{\omega}$ for all $a \in M$, we can suppose that $x_i \geq 0$ for all $i$ with no harm. Take $\varepsilon > 0$ and let $y_i=x_i 1_{[\varepsilon^{-1/2},+\infty)}(x_i)$ and $z_i=x_i-y_i$. Then the net $(z_i)_{i \in I}$ is bounded in $M$ so that $z=(z_i)^{\omega}$ defines an element of $M^{\omega}_{GR}$. Since $e\alpha  = 0$, we get that $\alpha=(\widehat{x_i})^\omega$ is orthogonal to $(\widehat{z_i})^\omega=\widehat{1}^\omega z$. Since $z_i x_i=z_i^{2}$, this implies that $(\widehat{z_i} )^\omega=0$. Hence $\alpha=(\widehat{x_i} )^\omega=(\widehat{y_i} )^\omega$. Moreover, since $\mathrm{supp}(y_i) \leq 1_{[\varepsilon^{-1/2},+\infty)}(x_i) \leq \varepsilon x_i^{2} $, we have $ \tau(\mathrm{supp}(y_i)) \leq \varepsilon ||x_i ||_2^{2} = \varepsilon$ for all $i \in I$.  Now we prove the following claim:
\begin{claim}
For every finite family  $a_1,\dots, a_n \in M^{+}$ and every $\varepsilon > 0$ there exists a non-zero projection $p \in M$ with $\tau(p) \leq \varepsilon$ such that:
\[ \sum_k ||pa_k - a_k p ||_2^2 \leq \varepsilon \tau(p) \]
\end{claim}

Indeed, take a finite family $a_1,\cdots, a_n \in M^{+}$. Since $\langle x^{\omega} \alpha, \alpha \rangle =\tau(x)$ for all $x \in M$, we have $||a_k^{\omega}\alpha ||^{2}_2=||a_k||^{2}_2$ for all $k$. Moreover we have $a_k^{\omega}\alpha=\alpha a_k^{\omega}$ for all $k$. Hence taking $x:=y_i$ for $i$ large enough, we get a non-zero $x \in M^{+}$ such that:
\begin{itemize}
\item $\sum_k ||a_k x - x a_k ||_2^2 \leq \varepsilon \sum_k || a_k x ||_2^2$.
\item $\tau(\mathrm{supp}(x)) \leq \varepsilon$.
\end{itemize}
Let $a=(a_1,\dots, a_n) \in (M^{(n)})^{+}$. Observe that $|| x a ||^2=\sum_k ||x a_k||^2$ and $|| x a - a x ||_2^2=\sum_k ||x a_k - a_k x ||_2^2$. Thus we have $|| x a - a x||_2^2 \leq \varepsilon || x a ||_2^2$. Hence by Lemma \ref{spike}, we can take $c > 0$ such that the projection $p=1_{[c,+\infty)}(x)$ is non-zero and satisfies $||p a - a p||_2^2 \leq 4 \sqrt{\varepsilon}||p a ||_2^{2} \leq 4 \sqrt{\varepsilon} ||a||_\infty^{2} \tau(p)$. Moreover $\tau(p) \leq \tau(\mathrm{supp}(x)) \leq \varepsilon$. Since $\varepsilon$ is arbitrary, we have proved the claim. Observe also that the claim holds not only for $M$ but also for all of its corners $qMq$ because we can repeat the same argument using the vector $q^{\omega} \alpha=q^{\omega} \alpha q^{\omega} \in L^{2}(qMq)^{\omega}$ ($q^{\omega} \alpha$ is non-zero because $\langle q^{\omega} \alpha, \alpha \rangle=\tau(q) \neq 0$).

Finally, we will patch the projections of the claim together in order to construct a central net of projections $(p_i)_{i \in I}$ in $M$ such that $\tau(p_i)=\frac{1}{2}$ and this will contradict the fullness of $M$. So take a finite family $a_1,\cdots, a_n \in M^{+}$ and $\varepsilon > 0$. We want to construct a projection $p \in M$ such that $\tau(p)=\frac{1}{2}$ and $\sum_k ||pa_k - a_k p ||_2^2 \leq \varepsilon$.  Let $a=(a_1,\dots, a_n) \in (M^{(n)})^+$. Let $R$ be the set of all projections $p \in \mathcal{P}(M)$ such that $||ap-pa||_2^{2} \leq \varepsilon \tau(p)$ and $\tau(p) \leq \frac{1}{2}$. The poset $R$ is clearly inductive because $R$ is closed in $\mathcal{P}(M)$ and thus for any increasing net $p_i \in R, i \in I$ we have $\bigvee_i p_i \in R$. By Zorn's lemma, take $p$ a maximal element in $R$. Let $\delta=\frac{1}{2}- \tau(p) \geq 0$. Suppose that $\delta > 0$. Let $b = p^\perp a p^\perp$. Since the claim holds for all the corners of $M$, we can find a non-zero projection $q \in p^\perp Mp^\perp$ such that $||qb-bq||_2 \leq \varepsilon \tau(q)$ and $\tau(q) \leq \delta$. Let $p'=p+q$. Then we check easily that $||p' a - a p'||_2^{2} \leq \varepsilon \tau(p')$ and $\tau(p') \leq \frac{1}{2}$. Thus $p' \in R$ and this contradicts the maximality of $p$. Hence $\delta=0$ which means that we have found a projection $p$ such that $\sum_k ||pa_k-a_kp||_2^{2} \leq \varepsilon$ and $\tau(p)=\frac{1}{2}$ as we wanted.
\end{proof}

Now, we want to prove the spectral gap theorem for full type $\mathrm{III}$ factors. First, we need a lemma to deal with the case where the ultraproduct vector we construct lies in $L^{2}(M^{\omega})$. 

\begin{lem} \label{full_unique}
Let $M$ be a full $\sigma$-finite factor. Let $\varphi$ be a faithful normal state on $M$ and let $\xi = \varphi^{1/2}$. Let $\omega$ by any cofinal ultrafilter on any directed set $I$. Suppose that $\alpha \in L^{2}(M^{\omega})$ satisfies $\alpha \xi^{\omega}=\xi^{\omega}\alpha$ (the equality holds in $L^{1}(M^{\omega})=(M^{\omega})_*$) and 
\[ \forall a,b \in M, \, a\xi=\xi b \Rightarrow a^\omega \alpha= \alpha b^\omega \]
Then $\alpha\in \C \xi^\omega$.
\end{lem}
\begin{proof}
We suppose that $||\alpha || =1$ without loss of generality. Let $p$ be the right support of $\alpha$ (the smallest projection $p \in M^\omega$ such that $\alpha p=p$). For every $\varphi$-analytic element $b \in M$ there exists $a \in M$ such that $a\xi=\xi b$. Hence $a^\omega\alpha = \alpha b^\omega$. This means that $\alpha b^\omega p= \alpha b^{\omega}$ and therefore $pb^\omega p=pb^\omega $. Since the same is true for $b^*$ we get $b^\omega p=pb^\omega$ and since the $\varphi$-analytic elements are dense in $M$, we conclude that $p \in M'\cap M^{\omega}$. Since $\alpha \xi^{\omega}=\xi^{\omega}\alpha$, we have also that $p$ is in the centralizer of $\varphi^{\omega}$. Hence $p \in M_\omega = \C$ which means that $p=1$. Thus $\psi= \alpha^*\alpha \in (M^{\omega})^+_*$ is a faithful normal state. Let $a \in M$ be a $\varphi$-analytic element and take $b$ and $c$ in $M$ such that $a\xi=\xi b$ and $b\xi = \xi c$ so that $a\varphi=\varphi c$. Then we have $b^* \xi = \xi a^*$. Hence we have $(b^*)^\omega \alpha = \alpha (a^*)^\omega$ and thus $a^\omega \alpha^*=\alpha^* b^\omega$. Thus we have $a^\omega \psi = a^\omega \alpha^*\alpha = \alpha^* b^\omega \alpha=\alpha^*\alpha c^\omega=\psi c^\omega$. From this and \cite[Corollary 3.4]{haagerup1979operator}, we conclude that the graph of the analytic generator of $\sigma^\varphi$ is contained in the graph of the analytical generator of $\sigma^\psi$. By \cite[Lemma 4.4]{haagerup1979operator}, this implies that $\sigma_t^\psi(x)=\sigma_t^\varphi(x)=\sigma_t^{\varphi^{\omega}}(x)$ for all $x \in M$ and all $t \in \R$. Hence the Connes-Radon-Nikodym cocycle $u_t=\psi^{{\rm i} t}(\varphi^{\omega})^{-{\rm i} t}$ is in $M' \cap M^\omega$. Moreover, since $\alpha \xi^{\omega}=\xi^{\omega}\alpha$ we have $\psi \varphi^{\omega}=\varphi^{\omega} \psi$ and therefore $\psi^{it}\varphi^{\omega}=\varphi^{\omega} \psi^{it}$ for all $t$ by Lemma \ref{strongly_commute}. Hence, we have $u_t \varphi^{\omega}=\varphi^{\omega}u_t$. Altogether, we get that $u_t \in M_\omega=\C$ for all $t \in \R$ and therefore $\psi=\varphi^\omega$. By polar decomposition, we thus have $\alpha = u |\alpha |=u \xi^{\omega}$ for some $u \in \mathcal{U}(M^\omega)$. Now take $x,y \in M$ such that $x \xi =\xi y$. Then we also have $x^\omega \alpha = \alpha y^\omega$. Therefore we get $x^\omega u \xi^\omega=u \xi^\omega y^\omega=u x^\omega \xi^\omega$ which means that $x^\omega u=ux^\omega$. Since the $\varphi$-analytic elements are dense in $M$ we get that $u \in M' \cap M^\omega$. And since $\alpha \xi^{\omega}=\xi^{\omega}\alpha$ we get $u \varphi^{\omega}=\varphi^{\omega}u$ so that $u \in M_\omega=\C$ and we are done.
\end{proof}

We also need a type $\mathrm{III}$ version of the maximality argument.

\begin{lem} \label{maximality}
Let $M$ be a $\sigma$-finite type $\mathrm{III}$ factor. Suppose that for every finite family of vectors $\xi_1,\dots, \xi_n \in L^{2}(M)^+$, every $\varepsilon > 0$ and every neighborhood $\mathcal{V}$ of $0$ in $\mathcal{P}(M)$,  there exists a non-zero projection $p \in \mathcal{V}$ such that:
\[ \sum_k ||p\xi_k - \xi_k p ||^2 \leq \varepsilon \sum_k ||p \xi_k||^2 \]
Then $M$ is not full.
\end{lem}
\begin{proof}
Let $\varphi$ be a faithful normal state on $M$. We will show that there exists a net of projections $p_i \in \mathcal{P}(M), i \in I$ such that $\varphi(p_i)=\frac{1}{2}$ and $I(p_i,\xi) \rightarrow 0$ for all $\xi \in L^{2}(M)^{+}$. This will prove that $M$ is not full because the net $(p_i)_{i \in I}$ will be centralizing thanks to Proposition \ref{skew}.$(1)$. So let $\xi_1,\dots,\xi_n \in L^2(M)^+$ be a finite family of vectors in the positive cone and let $\varepsilon > 0$. We want to construct a projection $p \in M$ such that $I(p,\xi_k) \leq \varepsilon$ for all $k$ and $\varphi(p)=\frac{1}{2}$. Let $\xi=(\xi_1,\dots, \xi_n) \in L^2(M^{(n)})^+$. Let $R$ be the set of all projections $p \in \mathcal{P}(M)$ such that $I(p,\xi) \leq \varepsilon ||p\xi ||^2$ and $\varphi(p) \leq \frac{1}{2}$. The poset $R$ is clearly inductive because $R$ is closed in $\mathcal{P}(M)$ and thus for any increasing net $p_i \in R, i \in I$ we have $\bigvee_i p_i \in R$. By Zorn's lemma, take $p$ a maximal element in $R$. Let $\delta=\frac{1}{2}- \varphi(p) \geq 0$. Suppose that $\delta > 0$. Let $\eta = p^\perp \xi p^\perp$. Since $M$ is a $\sigma$-finite type $\mathrm{III}$ factor and $p^{\perp} \neq 0$ we have $M \simeq p^\perp Mp^\perp$ and thus $p^\perp Mp^\perp$ also satisfies the assumption of the lemma. Therefore, we can find a non-zero projection $q \in p^\perp Mp^\perp$ such that $I(q,\eta) \leq \varepsilon ||q \eta ||^2$ and $\varphi(q) \leq \delta$. Let $p'=p+q$. Then we have  $I(p',\xi) = I(p,q^{\perp}\xi q^{\perp}) + I(q,p^{\perp} \xi p^{\perp})$. Since $I(p,q^{\perp}\xi q^{\perp}) \leq I(p,\xi) \leq \varepsilon ||p \xi ||^{2}$ and $I(q,p^{\perp} \xi p^{\perp})=I(q,\eta) \leq \varepsilon ||q \eta ||^{2} \leq \varepsilon ||q \xi ||^{2}$ we get $I(p',\xi) \leq \varepsilon ||p\xi ||^2 + \varepsilon ||q \xi ||^2=\varepsilon ||p'\xi||^2$. Moreover, by construction, we have $\varphi(p') \leq \frac{1}{2}$. Thus $p' \in R$ and this contradicts the maximality of $p$. Hence $\delta=0$ which means that we have found a projection $p$ such that $\sum_k I(p,\xi_k) = I(p,\xi) \leq \varepsilon || p\xi ||^2$ and $\varphi(p)=\frac{1}{2}$. Since $\varepsilon$ was arbitrary, we are done.
\end{proof}

Now we are ready to prove our main theorem.

\begin{proof}[Proof of Theorem \ref{main}]
Suppose by contradiction that the conclusion of the theorem is not true. We will prove the hypothesis of Lemma \ref{maximality} and this will contradict the fullness of $M$. So take a family of non-zero vectors $\xi_1, \dots, \xi_n \in L^{2}(M)^{+}$.  Let $\varphi=\xi_1^{2}+\dots + \xi_n^{2} \in M_*^{+}$. We want to construct a projection $p$ as in Lemma \ref{maximality} and for this we can suppose that $\varphi$ is faithful (otherwise just take $p$ a small projection in $\mathrm{supp}(\varphi)^{\perp}$). Also, without loss of generality, we can suppose that $\varphi$ is a state. Now, let $\xi=\varphi^{1/2}$. Take a net $(x_i)_{i \in I}$ in $M$ such that:
\begin{itemize}
\item $|| x_i  ||_\varphi=||x_i \xi ||=1$ for all $i$.
\item $\varphi(x_i)=\langle x_i \xi, \xi \rangle =0$ for all $i$.
\item $x_i \eta - \eta x_i \rightarrow 0$ for every $\eta \in L^{2}(M)^{+}$ with $\eta^{2} \leq \varphi$.
\end{itemize}
Then, since $\xi M \cap M \xi$ is the linear span of $\{ \eta \in L^2(M)^{+} \mid \eta^2 \leq \varphi \}$, we have in fact $x_i \eta - \eta x_i \rightarrow 0$ for all $\eta \in \xi M \cap M \xi$. Now, let $\omega$ be a cofinal ultrafilter on $I$. For all $\eta \in \xi M \cap M \xi$ we have $(x_i \eta)^\omega=(\eta x_i)^\omega$ by construction (note that even though the net $(x_i)_{i \in I}$ is not bounded in $M$, the net $(x_i \eta)_{ i \in I}$ is bounded in $L^2(M)$ for all $\eta \in \xi M$). Define $\alpha =(x_i \xi)^\omega=(\xi x_i)^{\omega} \in L^2(M)^\omega$. Then we have
\begin{itemize}
\item $||\alpha ||=1$
\item $\alpha \perp \xi^{\omega}$
\item $\alpha \xi^\omega=(\xi x_i \xi)^\omega = \xi^\omega \alpha$ (where we use \cite{raynaud2002ultrapowers}[Theorem 5.1])
\item $a^{\omega}\alpha=\alpha b^{\omega}$ for all $a,b \in M$ such that $a\xi=\xi b$. 
\end{itemize} 
The last property comes from the fact that $a^{\omega} \alpha=(\eta x_i)^\omega = (x_i \eta)^\omega = \alpha b^{\omega}$ where $\eta=a\xi=\xi b \in \xi M \cap M \xi$.

Let $e$ be the support of $\xi^\omega$ in $M^\omega_{\mathrm{GR}}$ so that $M^\omega = e(M^\omega_{\mathrm{GR}})e$ and $L^2(M^\omega)=eL^2(M)^\omega e$. Since $\alpha \xi^\omega= \xi^\omega \alpha$ we have $\alpha':=e\alpha=\alpha e \in L^2(M^{\omega})$ and $\alpha'\xi^\omega=\xi^\omega \alpha'$. And since $x^\omega e=ex^\omega$ for all $x \in M$, we have $a^{\omega}\alpha'=\alpha' b^{\omega}$ for all $a,b \in M$ such that $a\xi=\xi b$.  Since $M$ is full, we conclude by Lemma \ref{full_unique} that $\alpha'=\lambda \xi^\omega$ for some $\lambda \in \C$. But $\lambda = \langle \alpha', \xi^\omega \rangle = \langle e \alpha , \xi^\omega \rangle = \langle \alpha , e\xi^\omega \rangle = \langle \alpha, \xi^\omega \rangle = 0$. Hence $\alpha'=e\alpha=\alpha e =0$.

Now using the fact that $e\alpha=\alpha e=0$, we modify the $x_i$'s to make their supports very small. Take $\varepsilon > 0$ and let $y_i=x_i 1_{[\varepsilon^{-1},+\infty)}(x_i^{*}x_i)=1_{[\varepsilon^{-1},+\infty)}(x_i x_i^{*})x_i$ and $z_i=x_i-y_i$. Then the net $(z_i)_{i \in I}$ is bounded in $M$ so that $z=(z_i)^{\omega}$ defines an element of $M^{\omega}_{GR}$. Since $\alpha e =e\alpha  = 0$, we know that $\alpha=(x_i \xi)^\omega=(\xi x_i)^\omega$ is orthogonal to $(z_i \xi)^\omega=z \xi^\omega$ and $(\xi z_i)^\omega=\xi^\omega z$. Since $z_i^*x_i=z_i^*z_i$ and $x_i z_i^*=z_i z_i^*$, this implies that $(z_i \xi)^\omega=(\xi z_i)^\omega=0$ and therefore $(z_i \eta)^\omega=(\eta z_i)^\omega=0$ for all $\eta \in \xi M \cap M \xi$. Hence $(y_i \eta)^\omega=(x_i \eta)^\omega$ and $(\eta y_i)^\omega=(\eta x_i)^\omega$ for all $\eta \in \xi M \cap M \xi$. In particular, we have $(y_i \eta)^\omega=(\eta y_i)^\omega$ for all $\eta \in \xi M \cap M \xi$. Moreover, since $\mathrm{supp}(|y_i|) \leq 1_{[\varepsilon^{-1},+\infty)}(x_i^{*}x_i) \leq \varepsilon x_i^*x_i $, we have $ \varphi(\mathrm{supp}(|y_i |)) \leq \varepsilon ||x_i ||_\varphi^{2} = \varepsilon$. Hence we have constructed a net $y_i \in M, i \in I$ such that $|| (y_i \xi)^{\omega} ||= || \alpha || =1$, $(y_i \xi_k)^\omega=(\xi_k y_i)^\omega$ for $k=1, \dots, n$, and $\varphi(\mathrm{supp}(|y_i|)) \leq \varepsilon$. Therefore, taking $x:=y_i$ for $i$ large enough, we can get an $x \in M$ such that:
\begin{itemize}
\item $\sum_k ||x \xi_k - \xi_k x ||^2 < \varepsilon ||x \xi||^2$.
\item $\varphi(\mathrm{supp}(|x|)) \leq \varepsilon$.
\end{itemize}
Let $\pi=(\xi_1,\dots, \xi_n) \in L^{2}(M^{(n)})^{+}$. Observe that $||x \xi ||^2=\sum_k ||x \xi_k||^2=||x\pi ||^2$ and $|| x \pi - \pi x ||^2=\sum_k ||x \xi_k - \xi_k x ||^2$. Thus we have $|| x \pi - \pi x||^2 \leq \varepsilon || x \pi ||^2$. By Proposition \ref{skew} (3), we get $|| \, |x| \pi - \pi |x| \, ||^2 \leq 2\varepsilon || \, |x| \pi ||^{2}$. Hence by Lemma \ref{spike}, we can take $c > 0$ such that the projection $p=1_{[c,+\infty)}(|x|)$ is non-zero and satisfies $||p \pi - \pi p||^2 \leq 4 \sqrt{2\varepsilon}||p \pi ||^{2}$. Moreover $\varphi(p) \leq \varphi(\mathrm{supp}(|x|)) \leq \varepsilon$. Since $\varepsilon$ is arbitrary, we have proved the assumption of Lemma \ref{maximality} so we conclude that $M$ is not full and this is a contradiction.
\end{proof}

Note that in Theorem \ref{main}, we are not free to choose $\varphi$ as we want, only its existence is guaranteed. Also, we could not show Theorem \ref{main} for $\mathrm{II}_\infty$ factors or even $\mathrm{I}_\infty$ factors. Hence, the best statement we could get in general is the following:

\begin{theo} \label{gap_general}
Let $M$ be a full factor. Then there exist a normal state $\varphi$ on $M$ and a family $\xi_1, \cdots, \xi_n \in L^{2}(M)^{+}$ with $\xi_k^{2} \leq \varphi$ for all $k$ such that for all $x \in M$ we have
\[ ||x-\varphi(x)||_\varphi^2 \leq \sum_k || x \xi_k-\xi_k  x ||^2  \]
If $M$ is of type $\mathrm{III}$ and $p \neq 0$ is any $\sigma$-finite projection, we can choose $\varphi$ such that $\mathrm{supp}(\varphi)=p$.

If $M$ is semifinite and $p \neq 0$ is any finite projection, we can choose $\varphi$ such that $\mathrm{supp}(\varphi)=p$ and $\varphi_{|pMp}$ is the unique tracial state of $pMp$.
\end{theo}
\begin{proof}
Take $p$ a projection as in the statement of the theorem and apply Theorem \ref{main} or Theorem \ref{gap_connes} to $pMp$. Then we get the desired the state $\varphi \in (pMp)_*=pM_*p$ and $\xi_1, \cdots, \xi_n \in L^2(pMp)^{+} \subset L^{2}(M)^{+}$ with $\xi_k^{2} \leq \varphi$ such that for all $x \in pMp$ we have
\[ ||x-\varphi(x)||_\varphi^2 \leq \sum_k || x \xi_k-\xi_k  x ||^2 \]
Now, for any $x \in M$ we have
\[ ||x-\varphi(x)||_\varphi^2=||p^{\perp}xp||_\varphi^2+||pxp-\varphi(pxp)||_\varphi^2 \]
Since $||p^{\perp}xp||_\varphi^2 \leq 4||x\varphi^{1/2}-\varphi^{1/2}x ||$ and 
\[ ||pxp-\varphi(pxp)||_\varphi^2 \leq \sum_k || (pxp) \xi_k-\xi_k (pxp) ||^2 \leq \sum_k || x \xi_k-\xi_k x ||^2 \]
we conclude that for any $x \in M$ we have
\[ ||x-\varphi(x)||_\varphi^2 \leq \sum_{k=1}^{n+4} || x \xi_k-\xi_k  x ||^2 \]
where $\xi_{n+i}:=\varphi^{1/2}$ for $1 \leq i \leq 4$.
\end{proof}

\section{Application to crossed products}

In this section we prove Theorem \ref{crossed_product} following the original proof of Jones \cite{jones1982central}. Then we apply it to prove Corollary \ref{ueda}.

We first need two technical lemmas. The first one strengthens slightly \cite{dixmier1971vecteurs}[Proposition 1].
\begin{lem} \label{invertible}
Let $M$ be any von Neumann algebra. Then the invertible elements are $*$-strongly dense in the unit ball of $M$.
\end{lem}
\begin{proof}
Let $Q=\{ x \in M \mid \| x \| \leq 1 \text{ and } x \text{ is invertible} \}$. Let $\overline{Q}$ be the $*$-strong closure of $Q$. It is easy to see by the functionnal calculus that any normal element in the unit ball of $M$ is in $\overline{Q}$. Hence, since $\overline{Q}$ is stable under multiplication, it is enough, thanks to the polar decomposition, to show that any partial isometry $u \in M$ belongs to $\overline{Q}$. By \cite{dixmier1971vecteurs}[Lemma 1], we can find an isometry $v \in M$ and a coisometry $w \in M$ such that $u=vw(u^*u)$. Therefore, since $\overline{Q}$ contains all projections and is stable by adjunction and multiplication, it is enough to show that any isometry $v \in M$ is in $\overline{Q}$. By the proof of\cite{dixmier1971vecteurs}[Lemma 2], we can find a sequence of projections $p_n \in M$ converging to $1$ strongly (hence $*$-strongly) and partial isometries $w_n \in M$ such that $vp_n+w_n$ is a unitary for all $n$. Then let $x_n = vp_n + \frac{1}{n}w_n \in Q$. We have $x_n \rightarrow v$ in the $*$-strong topology. Hence $v \in \overline{Q}$.
\end{proof}

The second lemma uses the factoriality in order to transform a convergence on a corner into a convergence modulo inner automorphisms. Recall that $\epsilon : \Aut(M) \rightarrow \mathrm{Out}(M)$ is the quotient map.

\begin{lem} \label{pre_gap}
Let $M$ be a factor. Let $p \in M$ be a non-zero projection. Let $\theta_i \in \Aut(M)$ be a net of automorphisms such that $\theta_i(\xi) \rightarrow \xi$ for all $\xi \in pL^2(M)p$. Then we have $\epsilon(\theta_i) \rightarrow 1$.
\end{lem}
\begin{proof}
After replacing $p$ by a smaller projection if necessary, we can suppose that $p$ is part of a system of matrix units $(e_{kl})_{k,l \in J}$ in $M$ with $e_{00}=p$. Let $x_i= \sum_{k} \theta_i(e_{k0})e_{0k}$ where the sum is $*$-strongly convergent. Then for all $\xi \in L^2(M)$ we have $\theta_i(\xi)x_i-x_i \xi  \rightarrow 0$. Indeed we have
\[ \theta_i(\xi)x_i=\sum_{k,l} \theta_i(e_{k0})\theta_i(e_{0k}\xi e_{l0})e_{0l} \]
and
\[ x_i \xi = \sum_{k,l} \theta_i(e_{k0})(e_{0k}\xi e_{l0} )e_{0l}  \]
Thus we get $\theta_i(\xi)x_i-x_i \xi  \rightarrow 0$ since by assumption we know that $\theta_i(e_{0k} \xi e_{l0}) \rightarrow e_{0k} \xi e_{l0}$ for all $k,l$.

We also note that $x_i^{*}x_i \leq 1$ and $x_i^*x_i \rightarrow 1$ strongly hence $|x_i| \rightarrow 1$ strongly and therefore we have $|x_i| \xi - \xi \rightarrow 0$ for all $\xi$. Similarly, we have $\theta_i^{-1}(|x_i^*|) \rightarrow 1$ strongly and therefore $\theta_i(\xi) |x_i^*| - \theta_i(\xi) \rightarrow 0$ for all $\xi$. By Lemma \ref{invertible}, we can find a net of invertible elements $y_i$ in the unit ball of $M$ such that $y_i - x_i \rightarrow 0$ and $\theta_i^{-1}(y_i)-\theta_i^{-1}(x_i) \rightarrow 0$ in the $*$-strong topology. Then we still have $\theta_i(\xi)y_i-y_i \xi  \rightarrow 0$, $|y_i| \xi - \xi \rightarrow 0$ and $\theta_i(\xi)|y_i^*|-\theta_i(\xi) \rightarrow 0$. Hence, by writing the polar decomposition $y_i=u_i |y_i|=|y_i^*| u_i$ with unitaries $u_i$, we get $\theta_i(\xi)u_i-u_i \xi \rightarrow 0$ for all $\xi \in L^2(M)$. This shows that $\epsilon(\theta_i) \rightarrow 1$.
\end{proof}

The following lemma is the key to the theorem. It generalizes \cite[Lemma 4]{jones1982central}.

\begin{lem} \label{neighborhood_spectral_gap} Let $M$ be a full factor and let $\varphi$ be a normal state as in Theorem \ref{gap_general}. For every neighborhood of the identity $\mathcal{N}$ in $\mathrm{Out}(M)$, there exists a family $\xi_1,\dots,\xi_n \in L^{2}(M)^{+}$ with $\xi_k^{2} \leq \varphi$ for all $k$ such that for all $x \in M$ and all $\theta \in \mathrm{Aut}(M)$ with $\epsilon(\theta) \notin \mathcal{N}$ we have  \[ ||x ||_\varphi^2 \leq \sum_k  || x \xi_k - \theta(\xi_k)x ||^2 \]
\end{lem}
\begin{proof}
Take a family $\alpha_1,\dots,\alpha_m \in L^{2}(M)^{+}$ with $\alpha_k^{2} \leq \varphi$ such that for all $x \in M$ we have \[ ||x-\varphi(x)||_\varphi^2 \leq \sum_k || x\alpha_k-\alpha_k  x ||^2  \]
Let $\mathcal{N}$ be a neighborhood of the identity in $\mathrm{Out}(M)$. Since $\{ \beta \in L^{2}(M)^{+} \mid \beta^{2} \leq \varphi \}$ spans a dense subspace of $pL^{2}(M)p$, then by Lemma \ref{pre_gap}, we can find $\beta_1, \dots, \beta_q \in L^2(M)^+$ such that for all $\theta \in \mathrm{Aut}(M)$ such that $\epsilon(\theta) \notin \mathcal{N}$ we have
\[ \sum_k  || \theta(\beta_k)-\beta_k ||^2 \geq 1 \]
Then by merging the $\alpha_k$ and the $\beta_k$ into a single family, by adding one more vector if necessary and then by renormalizing we can get a family $\xi_1,\dots,\xi_n \in L^2(M)^+$ such that

\begin{itemize}
\item $\varphi=\xi_1^{2} + \dots + \xi_n^{2}$.
\item For some constant $C > 0$ and for all $x \in M$ we have \[ ||x-\varphi(x)||_\varphi^2 \leq C\sum_k || x\xi_k-\xi_k  x ||^2  \]
\item For some $\varepsilon > 0$ and for all $\theta \in \mathrm{Aut}(M)$ with $\epsilon(\theta) \notin \mathcal{N}$ we have
\[ \sum_k  || \theta(\xi_k)-\xi_k ||^2 \geq \varepsilon \]
\end{itemize}

 Now, we will show that there exists a constant $B > 0$ such that for all $x \in M$ and all $\theta \in \mathrm{Aut}(M)$ with $\epsilon(\theta) \notin \mathcal{N}$ we have
\[ || x||_\varphi^2 \leq B \sum_k ||x \xi_k - \theta(\xi_k) x||^2 \]
If not, then we can find a net $x_i \in M, i \in I$ and a net $\theta_i \notin \epsilon^{-1}(\mathcal{N})$ such that $||x_i ||_\varphi= ||x_i\xi ||=1$ and $x_i\xi - \theta_i(\xi)x_i \rightarrow 0$ where $\xi=(\xi_1, \dots, \xi_k) \in L^{2}(M^{(n)})^+$. Thanks to Lemma \ref{invertible}, we can assume that every $x_i$ is invertible. Note that the net $(x_i)_{i \in I}$ is not necessarily bounded. Let $x_i=u_i |x_i|=|x_i^{*}|u_i, \; u_i \in \mathcal{U}(M)$ be the polar decomposition of $x_i$.  We have $x_i \xi - \theta_i(\xi)x_i \rightarrow 0$ and since $I(|x_i |, \xi) + I(|x_i^* |, \theta_i(\xi)) \leq ||x_i\xi - \theta_i(\xi)x_i||^2 $ we get $|x_i | \xi - \xi |x_i | \rightarrow 0$ and $|x_i^* | \theta_i(\xi) - \theta_i(\xi) |x_i^* | \rightarrow 0$. Hence, by the choice of $\xi_1,\dots, \xi_n$ we get $|| \, |x_i|-\varphi(|x_i|) \, ||_\varphi \rightarrow 0$ which means that $|x_i|\xi- \langle |x_i| \xi, \xi \rangle \xi \rightarrow 0$. Since $|| \, |x_i| \xi \, ||=1$ we get $\langle |x_i| \xi, \xi \rangle  =|| \langle |x_i| \xi, \xi \rangle \xi || \rightarrow 1$ and this implies that $(|x_i| -1) \xi \rightarrow 0$. Similarly, if we let $y_i=\theta_i^{-1}(|x_i^*|)$ then we have $y_i \xi -\xi y_i \rightarrow 0$ and thus $(y_i-1)\xi \rightarrow 0$ by the same argument. Hence, by composing with $\theta_i$, we get $(|x_i^*|-1) \theta_i(\xi) \rightarrow 0$ and by taking the adjoint we get $\theta_i(\xi)(|x_i^{*}|-1) \rightarrow 0$. Therefore, we get $(x_i - u_i) \xi=u_i(|x_i|-1)\xi \rightarrow 0$ and $\theta_i(\xi)(x_i-u_i)=\theta_i(\xi)(|x_i^{*}|-1)u_i \rightarrow 0$. Combining this with $x_i \xi - \theta_i(\xi)x_i \rightarrow 0$, we get $u_i \xi - \theta_i(\xi)u_i \rightarrow 0$. This is a contradiction because by assumption we have $ ||u_i \xi -\theta_i(\xi)u_i ||^{2} = \sum_k || \xi_k -(\mathrm{Ad}(u_i^{*}) \circ \theta_i)(\xi_k) ||^{2} \geq \varepsilon$.
\end{proof}

Now, we are ready to prove Theorem \ref{crossed_product}. In fact, we prove the following more precise statement:

\begin{theo}
Let $M$ be a full factor and $G$ a discrete group. Let $\sigma : G \rightarrow \mathrm{Aut}(M)$ be an outer action whose image in $\mathrm{Out}(M)$ is discrete. Let $\varphi$ be a normal state on $M$ as in Theorem \ref{gap_general}. Then we can find a family $\xi_1,\dots,\xi_n \in L^{2}(M)^{+}$ with $\xi_k^{2} \leq \varphi$ such that for all $x \in M \rtimes_\sigma G$ we have
\[ ||x-\varphi(x)||_\varphi^{2} \leq \sum_k ||x \xi_k - \xi_k x ||^{2} \]
where we identifiied $L^{2}(M)$ with a subspace of $L^{2}(M \rtimes_\sigma G)$ and extended $\varphi$ to $M \rtimes_\sigma G$ by using the canonical conditional expectation $E : M \rtimes_\sigma G \rightarrow M$.

In particular, $M \rtimes_\sigma G$ is a full factor.
\end{theo}
\begin{proof}
Let $\mathcal{N}$ be a neighborhood of the identity in $\mathrm{Out}(M)$ such that $\sigma(G) \cap \epsilon^{-1}(\mathcal{N})=\{ \mathrm{id} \}$ and take  $\xi_1, \dots, \xi_n$ as in Lemma \ref{neighborhood_spectral_gap}. Take $x \in M \rtimes_\sigma G$ and let $E$ be the canonical normal conditional expectation of $M \rtimes_\sigma G$ onto $M$. For every $g \in G$, let $x^g=E(u_g^*x)$. Then we have
\[ x \xi_k - \xi_k x = \sum_{ g \in G} u_g x^g \xi_k - \sum_{ g \in G} \xi_k u_g x^g=\sum_{g \in G} u_g(x^g \xi_k - \sigma_{g^{-1}}(\xi_k) x^g) \]
Hence we get
\[ \sum_k ||x \xi_k -\xi_k x ||^2 = \sum_{g \in G} \sum_k ||x^g \xi_k - \sigma_{g^{-1}}(\xi_k) x^g||^2 \geq \sum_{g \in G \setminus \{ \mathrm{id} \} }||x^g||_\varphi^2= ||x - E(x) ||_\varphi^2 \]
But we also have
\[ || E(x)-\varphi(x)||_\varphi^{2} \leq \sum_k ||E(x) \xi_k - \xi_k E(x) ||^{2} \leq \sum_k ||x \xi_k - \xi_k x ||^{2} \]
Hence, adding both inequalities we get
\[ ||x- \varphi(x)||_\varphi^{2} =|| E(x)-\varphi(x)||_\varphi^{2} + ||x - E(x) ||_\varphi^2 \leq 2 \sum_k ||x \xi_k - \xi_k x ||^{2} \]
as we wanted. This inequality implies that $p(M \rtimes_\sigma G)p$ is full where $p=\mathrm{supp}(\varphi) \in M$. Hence $M \rtimes_\sigma G$ itself is full.
\end{proof}

Before we prove Corollary \ref{ueda}, we first recall that every automorphism $\alpha$ of a von Neumann algebra $M$ extends naturally to an automorphism $\widetilde{\alpha}$ of $c(M)$, called the \emph{quantized modulus} of $\alpha$ \cite[Definition $\mathrm{XII}.6.12$]{TakesakiII}. The map $\mathrm{Aut}(M) \ni \alpha \mapsto \widetilde{\alpha} \in \mathrm{Aut}(c(M))$ is a continuous group homomorphism and if there exists $u \in \mathcal{U}(c(M))$ such that $\alpha(x)=uxu^*$ for all $x \in M$ then we have $\widetilde{\alpha}=\mathrm{Ad}(u)$  \cite[Lemma $\mathrm{XII}.6.14$]{TakesakiII}. 

\begin{proof}[Proof of Corollary \ref{ueda}]
The easy part is the only if direction and was already proved in \cite[Theorem 3.2 and Corollary 3.4]{shlyakhtenko2004classification}. We give a different proof. Suppose that $c(M)$ is full. Let $u_i \in \mathcal{U}(M), i \in I$ be a net of unitaries such that $\mathrm{Ad}(u_i) \rightarrow \mathrm{Id}$ in $\mathrm{Aut}(M)$. Then by the continuity of the quantized modulus we have that $\mathrm{Ad}(u_i) \rightarrow \mathrm{Id}$ in $\mathrm{Aut}(c(M))$. Since $c(M)$ is full, we conclude that there exists $z_i \in \mathbb{U}$ such that $z_iu_i \rightarrow 1$. Hence $M$ is full. Now we show that $\delta : \R \rightarrow \mathrm{Out}(M)$ is a homeomorphism on its range. Take a net $t_i \in \R, i \in I$ such that $\delta(t_i) \rightarrow \mathrm{Id}$. We have to show that $t_i \rightarrow 0$. Take a faithful normal state $\phi$ on $M$. Then there exists a net $u_i \in \mathcal{U}(M), i \in I$ such that $\mathrm{Ad}(u_i) \circ \sigma_{t_i}^\phi \rightarrow \mathrm{Id}$ when $i \rightarrow \infty$. Let $w_i = u_i \phi^{{\rm i} t_i} \in \mathcal{U}(c(M))$. Since the quantized modulus of $\mathrm{Ad}(u_i) \circ \sigma_{t_i}^\phi$ is $\mathrm{Ad}(w_i)$, we know, by continuity, that $\mathrm{Ad}(w_i) \rightarrow \mathrm{Id}$. Since $c(M)$ is full, this implies that there exists $z_i \in \mathbb{U}$ such that $z_i w_i \rightarrow 1$. For every $\lambda \in \R^*_+$, we have $\theta_\lambda(z_iw_i)=\lambda^{{\rm i} t_i}z_iw_i$. Since $z_iw_i  \rightarrow 1$ we get $\lambda^{{\rm i} t_i} \rightarrow 1$. Since this holds for every $\lambda \in \R^*_+$, it implies that $t_i \rightarrow 0$ (because $\R \ni t \mapsto (\lambda \mapsto \lambda^{{\rm i} t}) \in \widehat{\R^{*}_+}$ is an isomorphism of topological groups).

Now we prove the other direction. Suppose that $M$ is full and $\delta : \R \rightarrow \mathrm{Out}(M)$ is a homeomorphism on its range. Pick $T > 0$. Then $\delta(T\Z)$ is discrete in $\mathrm{Out}(M)$. Hence if $\phi$ is any faithful normal state on $M$, we know by Theorem \ref{crossed_product} that $M \rtimes_{\sigma^\phi} T\Z$ is full. Hence $c(M)$ is full by \cite[Lemma 6]{tomatsu2014characterization}.
\end{proof}

 \bibliography{database}
\end{document}